\documentclass[12pt]{amsart}
\usepackage{amsmath,amssymb,amsbsy,amsfonts,latexsym,amsopn,amstext,cite,
                                               amsxtra,euscript,amscd,bm,mathabx,mathrsfs, abraces}
\usepackage{url}
\usepackage[colorlinks,linkcolor=blue,anchorcolor=blue,citecolor=blue,backref=page]{hyperref}
\usepackage{color}
\usepackage{graphics,epsfig}
\usepackage{graphicx}
\usepackage{float} 
\usepackage[english]{babel}
\usepackage{mathtools}
\usepackage{todonotes}
\usepackage{url}
\usepackage[colorlinks,linkcolor=blue,anchorcolor=blue,citecolor=blue]{hyperref}

\usepackage[norefs,nocites]{refcheck}

\hypersetup{breaklinks=true}

\usepackage[english]{babel}
\begin{document}

\newtheorem{theorem}{Theorem}
\newtheorem{thm}{Theorem}
\newtheorem{claim}[thm]{Claim}
\newtheorem{cor}[thm]{Corollary}
\newtheorem{prop}[thm]{Proposition} 
\newtheorem{definition}[thm]{Definition}
\newtheorem{question}[thm]{Open Question}
\newtheorem{conj}[thm]{Conjecture}
\newtheorem{rem}[thm]{Remark}
\newtheorem{prob}{Problem}
\newtheorem*{theorem*}{Theorem}
\newtheorem*{prop*}{Proposition}
\renewcommand*{\thetheorem}{\Alph{theorem}}
\newtheorem{lem}[theorem]{Lemma}
\def\ccr#1{\textcolor{red}{#1}}
\def\cco#1{\textcolor{orange}{#1}}
\def\ccc#1{\textcolor{cyan}{#1}}

\newtheorem{ass}[thm]{Assumption}

\newtheorem{lemma}[thm]{Lemma}

\newcommand{\GL}{\operatorname{GL}}
\newcommand{\SL}{\operatorname{SL}}
\newcommand{\lcm}{\operatorname{lcm}}
\newcommand{\ord}{\operatorname{ord}}
\newcommand{\Tr}{\operatorname{Tr}}
\newcommand{\Span}{\operatorname{Span}}

\numberwithin{equation}{section}
\numberwithin{theorem}{section}
\numberwithin{thm}{section}
\numberwithin{table}{section}

\def\sssum{\mathop{\sum\!\sum\!\sum}}
\def\ssum{\mathop{\sum\ldots \sum}}
\def\dsum{\mathop{\sum \ \sum}}

\def\vol {{\mathrm{vol\,}}}
\def\squareforqed{\hbox{\rlap{$\sqcap$}$\sqcup$}}
\def\qed{\ifmmode\squareforqed\else{\unskip\nobreak\hfil
\penalty50\hskip1em\null\nobreak\hfil\squareforqed
\parfillskip=0pt\finalhyphendemerits=0\endgraf}\fi}

\newcommand{\rank}{\operatorname{rank}}

\def \balpha{\bm{\alpha}}
\def \bbeta{\bm{\beta}}
\def \bgamma{\bm{\gamma}}
\def \blambda{\bm{\lambda}}
\def \bchi{\bm{\chi}}
\def \bphi{\bm{\varphi}}
\def \bpsi{\bm{\psi}}
\def \bomega{\bm{\omega}}
\def \btheta{\bm{\vartheta}}
\def \bmu{\bm{\mu}}
\def \bnu{\bm{\nu}}

\newcommand{\bfxi}{{\boldsymbol{\xi}}}
\newcommand{\bfrho}{{\boldsymbol{\rho}}}

\def\cA{{\mathcal A}}
\def\cB{{\mathcal B}}
\def\cC{{\mathcal C}}
\def\cD{{\mathcal D}}
\def\cE{{\mathcal E}}
\def\cF{{\mathcal F}}
\def\cG{{\mathcal G}}
\def\cH{{\mathcal H}}
\def\cI{{\mathcal I}}
\def\cJ{{\mathcal J}}
\def\cK{{\mathcal K}}
\def\cL{{\mathcal L}}
\def\cM{{\mathcal M}}
\def\cN{{\mathcal N}}
\def\cO{{\mathcal O}}
\def\cP{{\mathcal P}}
\def\cQ{{\mathcal Q}}
\def\cR{{\mathcal R}}
\def\cS{{\mathcal S}}
\def\cT{{\mathcal T}}
\def\cU{{\mathcal U}}
\def\cV{{\mathcal V}}
\def\cW{{\mathcal W}}
\def\cX{{\mathcal X}}
\def\cY{{\mathcal Y}}
\def\cZ{{\mathcal Z}}
\def\Ker{{\mathrm{Ker}}}

\def\sA{{\mathscr A}}

\def\NmQR{N(m;Q,R)}
\def\VmQR{\cV(m;Q,R)}

\def\Xm{\cX_m}

\def \A {{\mathbb A}}
\def \B {{\mathbb A}}
\def \C {{\mathbb C}}
\def \F {{\mathbb F}}
\def \G {{\mathbb G}}
\def \L {{\mathbb L}}
\def \K {{\mathbb K}}
\def \N {{\mathbb N}}
\def \Q {{\mathbb Q}}
\def \R {{\mathbb R}}
\def \Z {{\mathbb Z}}

\def \fA{\mathfrak A}
\def \fC{\mathfrak C}
\def \fL{\mathfrak L}
\def \fR{\mathfrak R}
\def \fS{\mathfrak S}

\def \fUg{{\mathfrak U}_{\mathrm{good}}}
\def \fUm{{\mathfrak U}_{\mathrm{med}}}
\def \fV{{\mathfrak V}}
\def \fG{\mathfrak G}
\def \f{\mathfrak G}

\def\e{{\mathbf{\,e}}}
\def\ep{{\mathbf{\,e}}_p}
\def\eq{{\mathbf{\,e}}_q}

 \def\\{\cr}
\def\({\left(}
\def\){\right)}
\def\fl#1{\left\lfloor#1\right\rfloor}
\def\rf#1{\left\lceil#1\right\rceil}

\def\Im{{\mathrm{Im}}}

\def \oF {\overline \F}

\newcommand{\pfrac}[2]{{\left(\frac{#1}{#2}\right)}}

\def \Prob{{\mathrm {}}}
\def\e{\mathbf{e}}
\def\ep{{\mathbf{\,e}}_p}
\def\epp{{\mathbf{\,e}}_{p^2}}
\def\em{{\mathbf{\,e}}_m}

\def\Res{\mathrm{Res}}
\def\Orb{\mathrm{Orb}}

\def\vec#1{\mathbf{#1}}
\def \va{\vec{a}}
\def \vb{\vec{b}}
\def \vc{\vec{c}}
\def \vs{\vec{s}}
\def \vu{\vec{u}}
\def \vv{\vec{v}}
\def \vw{\vec{w}}
\def\vlam{\vec{\lambda}}
\def\flp#1{{\left\langle#1\right\rangle}_p}

\def\mand{\qquad\mbox{and}\qquad}
\theoremstyle{definition}
\newtheorem*{theorem A}{Theorem A}
\newtheorem*{theorem B}{N\"olker's Theorem}
\newtheorem{proposition}{Proposition}[section]
\newtheorem{corollary}[theorem]{Corollary}
\newtheorem{problem}{Problem}
\newtheorem {conjecture}{Conjecture}
\theoremstyle{remark}
\newtheorem{remark}{Remark}[section]
\theoremstyle{remark}
\newtheorem{remarks}{Remarks}


\title[Coefficients of binary cyclotomic polynomials]
{On nonzero coefficients of binary cyclotomic polynomials}

\author[I. E. Shparlinski] {Igor E. Shparlinski}\thanks{During the preparation of this work, I.S. was supported in part by the Australian Research Council Grant~DP230100534}
\address{School of Mathematics and Statistics, University of New South Wales, Sydney NSW 2052, Australia}
\email{igor.shparlinski@unsw.edu.au}

\author [L. P. Wijaya]{Laurence P. Wijaya}
\address{Department of Mathematics, University of Kentucky, 715 Patterson Office Tower, Lexington, KY 40506, USA}
\email{laurence.wijaya@uky.edu}


\begin{abstract} Let $\vartheta(m)$ be the number of nonzero coefficients in the $m$-th cyclotomic polynomial. For real $\gamma > 0$ and $x \ge 2$ we define 
$$H_{\gamma}(x)=\#\left\{m:~m=pq \le x, \ p<q\text{ primes }, \
\vartheta(m)\le m^{1/2+\gamma}\right\},
$$
and show that for any fixed $\eta> 0$, uniformly over $\gamma$ with 
$$9/20+\eta \le \gamma\le 1/2 -\eta,
$$ 
we have  an asymptotic formula 
$$
H_{\gamma}(x)\sim C(\gamma)x^{1/2+\gamma}/ \log x, \qquad x \to \infty, 
$$
where  $C(\gamma)> 0$ is an explicit constant depending only on $\gamma$. 
This extends the previous result of {\'E}.~Fouvry (2013), which has $12/25$ instead  of $9/20$. 
This improvement is based on new ingredient including 
 work of W.~Duke, J.~Friedlander  and H.~Iwaniec  (1997). 
\end{abstract}

\maketitle

\tableofcontents

\section{Introduction}

\subsection{Motivation and main result}  
 For $m\ge 1$ integer, let 
 $$
\Phi_m(x)=\prod_{\substack{j=1\\\gcd(j,m)=1}}^{m}(x-\em(j))
$$
 be the $m$-th cyclotomic polynomial,  where 
$$\e(z) = \exp(2 \pi z) \mand \em(x)=\e(x/m).
$$ 
It is easy to see that cyclotomic polynomials are monic and have integer coefficients.

There is an extensive literature dedicated to studying the coefficients of cyclotomic polynomials, 
including their size, number of nonzero coefficients, gaps between non-zero coefficients and 
several others topics, see~\cite{AAHL, B12,CCLMS,F13,KMSZ} and references therein as well as
 exhaustive surveys by Herrera-Poyatos and Moree~\cite{H-PM} and by Sanna~\cite{San}. 

In particular, it is natural to study the sparsity of cyclotomic polynomials. Thus,  we define 
$\vartheta(m)$ as the  number  of nonzero coefficients of $\Phi_m$ and investigate  
various improvements of the following  trivial inequalities
$$
    1\le \vartheta(m)\leq \varphi(m)+1, 
$$
where $\varphi$ is the Euler totient function.

Here we are interested  in the special case where $m$ has the form $m = pq$, $p<q$ are primes. In this case,  $\Phi_{pq}$ is called a \textit{binary} cyclotomic polynomial. Binary, and similarly defined 
 \textit{ternary} cyclotomic polynomials, have many 
special properties and thus have attracted a lot of attention.  
For example, by a classical result of Migotti~\cite{Mig} 
the coefficients of binary cyclotomic polynomial  take values in the set $\{-1, 0, 1\}$. 

We recall that it has been known 
since the work of Gallot   and Moree~\cite{GalMor}, see also more recent results of~\cite{CCLMS,F13}, 
that the study of cyclotomic polynomials employs  rather 
advanced analytic techniques, such as bounds of {\it Kloosterman sums\/} and some other related exponential sums. Here we give further applications of these methods and in particular improve a result of Fouvry~\cite[Theorem~1]{F13}.  

More precisely, as in~\cite{F13}, for  real $\gamma > 0$ and $x \ge 2$ we define 
$$H_{\gamma}(x)=\#\left\{m:~m=pq \le x, \ p<q\text{ primes }, \
\vartheta(m)\le m^{1/2+\gamma}\right\}.
$$
It is also convenient to define 
\begin{equation}\label{eq:C-gamma}
        C(\gamma) =\frac{2}{1+2\gamma}\log \frac{1+2\gamma}{1-2\gamma}.
\end{equation}
Then, by~\cite[Theorem~1]{F13},  for every $\eta>0$, uniformly over $\gamma$ with 
\begin{equation}\label{eq:F-range}
\frac{12}{25}+\eta\le \gamma \le \frac{1}{2}-\eta,
\end{equation}
we have  an asymptotic formula 
\begin{equation}\label{eq:Asymp}
H_{\gamma}(x)\sim C(\gamma)\frac{x^{1/2+\gamma}}{\log x},  \qquad x \to \infty.
\end{equation}

The proof of~\eqref{eq:Asymp} in~\cite{F13} is based on bounds of analogues of Kloosterman 
sums over primes. 
Here we use a similar approach and show that the range~\eqref{eq:F-range} can be extended.  Namely, we have 

\begin{thm}\label{thm:H-gamma}
For any fixed $\eta > 0$, uniformly over $\gamma$ with 
$$
\frac{9}{20}+\eta\le \gamma \le \frac{1}{2}-\eta,
$$
the asymptotic formula~\eqref{eq:Asymp}  holds with $C(\gamma)$ given 
by~\eqref{eq:C-gamma}.
\end{thm}

We remark that 
$$
\frac{12}{25}=0.48 \mand \frac{9}{20} = 0.45.
$$

Theorem~\ref{thm:H-gamma}, follows  from repeating the argument of~\cite{F13} however with a stronger version of~\cite[Proposition~1]{F13}, 
which we formulate as  Proposition~\ref{prop: RPQ} and which we derive using some new argument. 

More precisely, our improvement of~\cite[Proposition~1]{F13} relies on new bounds of Kloosterman sums over primes from~\cite{Ir} and~\cite{KC}. 
However these 
new bounds alone are not sufficient for improving~\cite[Proposition~1]{F13}. This is because the result of 
both~\cite{Ir} and~\cite{KC}  are rather restrictive and require the parameters involved to be in certain ranges. 
To overcome this constraint, we also use one new technical 
tool,  which we derive from a result of Duke, Friedlander and Iwaniec~\cite{DFI}.

\subsection{Notation and conventions}

We adopt the Vinogradov notation `$\ll$',  that is,
$$
A\ll B~\Longleftrightarrow~A=O(B)~\Longleftrightarrow~|A|\le cB
$$
for some constant $c>0$ which  
may depend on the 
real positive parameters $\varepsilon$ and $\eta$, and is absolute otherwise.   

For a finite set $\cS$ we use $\# \cS$ to denote its cardinality.

Given two relatively prime integers $k$ and $\ell$, we write $\overline{k}_\ell $ for
the unique integer in $[1,\ell)$ such that $\overline{k}_\ell k\equiv 1 \mod \ell$. 

Throughout the paper, the letters  $p$ and $q$ always denote prime numbers.

Furthermore, we always assume that $x$ is a sufficiently large real number and $\eta$ and $\gamma\le 1/2-\eta$ are fixed. 
Next, we introduce the following quantities: 
$$
\cL = \log 2x \mand \xi 
=1+ \cL^{-1}, 
$$
and the function of real variable $t$: 
\begin{equation}
\label{eq:theta-0}
\rho(t) = \frac{1}{2}\(1+ t^{-1} 
- \sqrt{\(1+ t^{-1}\)^2- 2t^{-1} \(t^{1/2+\gamma}+1\)}\).
\end{equation}

We remark that we have changed the used in~\cite{F13} notation $\vartheta_0(t)$ to $\rho(t)$ 
in order to distinguish better from our main quantity $\vartheta(m)$. 
In fact we only need to know that 
\begin{equation}
\label{eq:theta-0-Asymp}
\rho(t) =\frac{t^{\gamma-1/2}}{2}+O(t^{2\gamma-1}), \qquad t \to \infty, 
\end{equation}
see Lemma~\ref{lem:theta-0} below.

For real  positive $u$ and $U$,  we write 
\begin{equation}\label{eq:approx}
u \approx  U
\end{equation}
 as an equivalent of 
$U  < u \le \xi U$. 

\subsection{Approach} 
We base our argument on  the ideas of Fouvry~\cite{F13}. 

In particular, we exploit a link between $H_{\gamma}(x)$ and bounds of exponential sums 
with relies 
on an explicit expression for  $\vartheta(pq)$ in terms of $p$ and $q$. 
More explicitly, by a classical result of Carlitz~\cite{Carl} (see also~\cite[Proposition~A]{F13}) 
 if  $p<q$ are two distinct primes, then 
$$
\vartheta(pq) = 2pq \overline{p}_q\overline{q}_p-1. 
$$ 
This can be rewritten in terms of $\overline{q}_p$ only 
as 
\begin{equation}\label{eq:theta=2pq}
\vartheta(pq) = 2 pq  \frac{\overline{q}_p}{p}\(1 + \frac{1}{pq}- \frac{\overline{q}_p}{p}\),
\end{equation}
see~\cite[Equation~(11)]{F13}.

To  establish Theorem~\ref{thm:H-gamma},   for real $P,Q\ge 1$,  we define 
by $R_{\gamma}(P,Q)$ the cardinality 
of the following auxiliary set (see~\cite[Equation~(16)]{F13}), 
    \begin{align*}
 R_{\gamma}&(P,Q)\\
 & \quad=\#\left\{(p,q) :~p<q,\ pq\le x, \ p \approx P, \ q \approx Q, 
    \ \frac{\overline{q}_p}{p}\le \rho(PQ)\right\}, 
    \end{align*}
where we recall our convention~\eqref{eq:approx}. This leads us to investigating the distribution of the 
fractions $\overline{q}_p/p$ which we approach via the {\it  Erd\H{o}s--Tur\'{a}n inequality\/} 
(see Lemma~\ref{lem:ET} below) and bounds 
of exponential sums. We also note that our  use  of the Erd\H{o}s--Tur\'{a}n inequality
instead of approximating the characteristic 
function of an interval via  a celebrated 
result of Vinogradov~\cite[Chapter~I, Lemma~12]{Vin} (as in~\cite{F13}),  leads to some technical simplifications.  

We also use the fact that  one only needs to consider  pairs $(P,Q)$ which satisfy
\begin{equation}\label{eq: PQ Cond}
    P\le \xi Q, \mand  \kappa_oQ^{(1-2\gamma)/(1+2\gamma)}\le P\le xQ^{-1}
\end{equation}
with 
$$
    \kappa_0=4^{-2/(1+2\gamma)} ,
 $$
 see~\cite[Equation~(21)]{F13}. 


Theorem~\ref{thm:H-gamma} is then a direct consequence of the following adjustment of~\cite[Proposition~1]{F13}
with  a wider range of $\gamma$. 

\begin{prop}\label{prop: RPQ}
For any fixed $\eta > 0$, uniformly over $\gamma$ with 
$$
\frac{9}{20}+\eta\le \gamma \le \frac{1}{2}-\eta,
$$
and $(P,Q)$ satisfying~\eqref{eq: PQ Cond}, we have 
\begin{align*} 
        R_{\gamma}(P,Q)=\frac{1}{2}(PQ)^{\gamma-1/2}&\(1+O(\mathcal{L}^{-1})\)R(P,Q) \\
        & \qquad +O\(x^{1/2+\gamma}\mathcal{L}^{-6} + Q\mathcal{L}^{-4}\), 
\end{align*}
where 
$$
R(P,Q) =  \dsum_{\substack{p\approx P, \ q\approx Q\\
p<q,\ pq\le x}} 1.
$$
\end{prop}

Hence, the rest of the paper is dedicated to the proof of Proposition~\ref{prop: RPQ}, which, as we have mentioned,  
immediately yields Theorem~\ref{thm:H-gamma} after substitution in the argument of~\cite{F13}.

\section{Preliminaries}

\subsection{Kloosterman sums over primes}
The proof for medium and large $P$ in~\cite[Section~4.3]{F13} 
(see also Section~\ref{sec:prelim} below for precise definition 
of these ranges) uses various bounds from~\cite{FS} on  Kloosterman sums over primes, 
which are our main tool as well. 
More precisely,  for real $z\ge y \ge 1$  and  a nonzero integer $a$
we are interested in the sums
$$
    S_p(a;y,z)=  \sum_{\substack{y\le q\le z\\ q \ne p}} \ep\(a \overline{q}_p\). 
$$

First, we record the following result of Korolev and Changa~\cite[Theorem~1]{KC}, which we use instead of 
the bound from~\cite[Theorem~3.2]{FS}, employed in~\cite{F13}. 

\begin{lemma}\label{lem:KC-bound}
If $p^{3/2}\ge z \ge p^{12/13}$, where $p$ is prime, we have for every $\varepsilon>0$
$$
    |S_p(a;y,z)|\le y^{15/16}p^{o(1)}, 
$$
    for every integer $a$ with $\gcd(a,p)=1$ and $1\le y\le z\le 2y$.
\end{lemma}

We also recall that~\cite[Theorem~1.5]{FKM} gives an nontrivial bound on very general sums 
over primes, but this does not seem to be of any help for our particular application.

Furthermore, similar to~\cite{F13}, we also use bounds on the sums $S_p(a;y,z)$ on average over $p$. 
However, here, instead of the bound from~\cite{FS},  we use a stronger result by Irving~\cite[Theorem~1.1]{Ir}.
In fact, we need it in slightly more general form when the limits of summation $y_p$ and $z_p$ are allowed to vary 
with $p$; one can easily check that the argument of  Irving~\cite{Ir} produces this extension without any changes.

\begin{lemma}\label{lem:Irving}  For  $P, Q \ge 1$, the inequality
$$
        \sum_{P<p\le 2P} \max_{\gcd(a,p)=1}|S_p(a;y_p,z_p)|\le(Q^{5/8}P^{5/4}+Q^{9/10}P+Q^{13/18}P^{7/6})P^{o(1)}
$$
    holds uniformly over $P^{3/2}\ge 2Q\ge (2P)^{2/3}$ and for any sequences $(y_p)_{P<p\le 2P}$ and
    $(z_p)_{P<p\le 2P}$ satisfying $Q\le y_p\le z_p\le 2Q$.
\end{lemma}

We note, that although the bounds of Lemmas~\ref{lem:KC-bound} and~\ref{lem:Irving}
improve the results from~\cite{FS}, the conditions on the summations ranges  in both of them are 
too restrictive for our application. To overcome this, we a use 
a modification of the following result of Duke, Friedlander and Iwaniec~\cite{DFI}. 

Namely, we recall that by a very special case of~\cite[Theorem~1]{DFI} we have
the following bound.

\begin{lemma}\label{lem:DFI-Original} For $P, Q \ge 1$, let $\alpha_p$ and $\beta_q$ be arbitrary sequences of complex numbers 
with 
$$
\alpha_p, \beta_q \ll 1, \quad P \le p \le 2P, \quad  Q \le q \le 2 Q, 
$$
and a positive integer $a\le PQ$, 
we have 
\begin{align*}
     \left |  \sum_{P<p\le 2P}  \sum_{\substack{Q\le q\le 2Q\\ q \ne p}} \alpha_p, \beta_q \ep\(a \overline{q}_p\)\right|& \\
     \le  (PQ)^{1/2+o(1)} &\(P^{1/2}+ Q^{1/2} + \min\{P,Q\}\).
\end{align*}
\end{lemma}  

We now combine Lemma~\ref{lem:DFI-Original} with  the standard completing technique, 
see~\cite[Section~12.2]{IwKow}, to derive the following bound.

\begin{lemma}\label{lem:DFI-bound}
For  $P, Q \ge 1$, the inequality
$$
        \sum_{P<p\le 2P} |S_p(a;y_p,z_p)|  \le  (PQ)^{1/2+o(1)}  \(P^{1/2}+ Q^{1/2} + \min\{P,Q\}\) 
$$
    holds uniformly over positive integers $a\le PQ$  and for any sequences $(y_p)_{P<p\le 2P}$ and
    $(z_p)_{P<p\le 2P}$ satisfying $Q\le y_p\le z_p\le 2Q$. 
\end{lemma}  

\begin{proof}Let $N = \rf{2Q}$. Observe that by the orthogonality of exponential functions, we have
$$
\frac{1}{N}\sum_{h=1}^{N}\sum_{y_p \le k \le z_p } \e_N\left(h(q-k)\right)=
  \begin{cases}
   1 & y_p \le q \le z_p, \\
  0 & \text{otherwise}.
  \end{cases}
$$
It follows that   
\begin{align*}
 S_p(a;y_p,z_p) &= \sum_{\substack{y\le q\le 2y\\ q \ne p}}  \ep\(a \overline{q}\)
 \frac{1}{N}\sum_{h=1}^{N} \sum_{y_p \le k \le z_p }  \e_N\left(h(q-k)\right)\\
&=\frac{1}{N} \sum_{h=1}^{N}  \sum_{y_p \le k \le z_p } \e_N\left(-hk\right )
 \sum_{\substack{Q\le q\le 2Q\\ q \ne p}}  \beta_{h,q}\ep\(a \overline{q}_p\), 
\end{align*}
where  $ \beta_{h,q} =   \e_N\(h q\)$.

We also note that for $1 \le h,y_p, z_p \le N$ we have
$$
\sum_{y_p \le k \le z_p } \e_N\left(hk\right )  \ll \frac{N}{\min\{h, N+1 - h\}},
$$
see~\cite[Equation~(8.6)]{IwKow}.   
Thus 
\begin{align*}
 S_p(a;y_p,z_p) & \ll    
 \sum_{h=1}^{N}  \frac{N}{\min\{h, N+1 - h\}}   
\left| \sum_{\substack{Q\le q\le 2Q\\ q \ne p}}  \beta_{h,q}\ep\(a \overline{q}_p\)\right| \\
&  =   
 \sum_{h=1}^{N}  \frac{1}{\min\{h, N+1 - h\}}      \alpha_{h,p}
 \sum_{\substack{Q\le q\le 2Q\\ q \ne p}}  \beta_{h,q}\ep\(a \overline{q}_p\) 
\end{align*} 
for some complex numbers  $\alpha_{h,p}$ with $| \alpha_{h,p}| =1$. 
Therefore, changing the order of summation, we obtain
\begin{align*}
 \sum_{P<p\le 2P} |S_p(a;y_p,z_p)|& \ll    
 \sum_{h=1}^{N}  \frac{1}{\min\{h, N+1 - h\}}   \\
& \qquad \qquad  \sum_{P<p\le 2P}   
 \sum_{\substack{Q\le q\le 2Q\\ q \ne p}}  \alpha_{h,p} \beta_{h,q}\ep\(a \overline{q}_p\).
 \end{align*}
Invoking Lemma~\ref{lem:DFI-Original}, we complete the proof. 
\end{proof}

\subsection{Exponential sums and discrepancy} 
Given a sequence of real numbers $\xi_1, \ldots, \xi_N \in [0,1)$, 
we define its  {\it discrepancy\/} $D_N$ as 
\begin{equation}
\label{eq:Discr}
D_N = \frac{1}{N}\sup_{0\le \gamma \le 1} \left |  \#\{1\le n\le N:~\xi_n\in [0, \gamma)\} - \gamma N \right | \,.
\end{equation}

We now recall  the classical Erd\H{o}s--Tur\'{a}n inequality,  which links the discrepancy and 
exponential sums (see, for instance,~\cite[Theorem~1.21]{DrTi} or~\cite[Theorem~2.5]{KuNi}).

\begin{lemma}
\label{lem:ET}
Let $\xi_n$, $n\in \N$,  be a sequence in $[0,1)$. Then for any $A\in \N$, the discrepancy $D_N$ given by~\eqref{eq:Discr} satisfies
$$
D_N \le 3 \left( \frac{1}{A+1} + \frac{1}{N}\sum_{a=1}^{A}\frac{1}{a} \left| \sum_{n=1}^{N} \e(a\xi_n) \right | \right) \,.
$$
\end{lemma} 

\subsection{Solutions to some quadratic equation}
One easily checks that the function $\rho(t)$ given by~\eqref{eq:theta-0} is a solution
to  polynomial equation
$$
2t \rho  \left( 1+\frac1t-\rho \right)-1=t^{1/2+\gamma}, 
$$ 
which in turn comes from the equation~\eqref{eq:theta=2pq}.
We need the following elementary result, see~\cite[Lemma~8]{F13}. 

\begin{lemma}  
\label{lem:theta-0} For any fixed $\eta > 0$, there exists some $T(\eta)$, such that the 
 function $\rho(t)$ is decreasing for $t>T(\eta)$ and for  $\gamma\le 1/2-\eta$ satisfies the 
 bound~\eqref{eq:theta-0-Asymp}.
\end{lemma}

In particular, since $\gamma\le 1/2-\eta$ for a fixed $\eta>0$, we see from Lemma~\ref{lem:theta-0}  that 
\begin{equation}\label{eq:thetaPQ}
\rho(PQ)=\frac12(PQ)^{\gamma-1/2}(1+O(\cL^{-3})),
\end{equation} 
provided that $PQ$ is large enough. 

\section{Proof of  Proposition~\ref{prop: RPQ}}

\subsection{Preliminaries}
\label{sec:prelim}
As in~\cite{F13} we divide the proof into three cases, depending on the size of $P$
\begin{itemize}
\item Small $P$:  $P < x^{1/3}\cL^{-100}$.
\item Medium $P$:   $x^{1/3}\cL^{-100}\le P \le(2x)^{2/5}$.
\item Large $P$: $(2x)^{2/5} < P \le x^{1/2}$. 
\end{itemize}

The treatment in~\cite[Section~4.1]{F13} of range of  \textit{small} $P$ does not  depend on the lower bound on  $\gamma$ and thus 
 gives the desired error term in the asymptotic formula of Proposition~\ref{prop: RPQ} 
 for any $\gamma<1/2 - \eta$.
Hence, this leaves us only with the two remaining  cases of \textit{medium} and \textit{large} $P$.

Furthermore, we may assume that 
\begin{equation}\label{eq: large PQ}
PQ\ge x\cL^{-12}, 
\end{equation}
see also~\cite[Equation~(28)]{F13}. 
Indeed, this is because if $\mathcal{E}(p,\rho(PQ))$ is the set of all congruence classes $s$ modulo  $p$ such that $\overline{s}_p\le \rho(PQ) p$, then
$$
    R_{\gamma}(P,Q)=\sum_{\substack{p\approx P\\y_p\le z_p}}\sum_{s\in \mathcal{E}(p,\rho(PQ))} (\pi(z_p;p,s)-\pi(y_p;p,s)),
$$
where, as usual, $\pi(x;k,a)$ denotes the number of primes $q \le x$ in the arithmetic progression $q \equiv a \mod k$. 
Using the trivial estimate 
$$\pi(z_p;p,s)-\pi(y_p;p,s)\ll Q/p+1\ll Q/p, 
$$ 
and recalling~\eqref{eq:theta-0-Asymp},
 we derive
$$
    R_{\gamma}(P,Q)\ll \rho(PQ) PQ  \ll (PQ)^{\gamma-1/2} PQ = (PQ)^{1/2+\gamma}.
$$
The main term in the asymptotic formula of Proposition~\ref{prop: RPQ} is trivially of size $O(PQ)^{1/2+\gamma}$.
So if $PQ \le x\cL^{-12}$, the error term Proposition~\ref{prop: RPQ} dominates other terms and 
Proposition~\ref{prop: RPQ} holds.  
Thus we can assume~\eqref{eq: large PQ}.

\subsection{Medium $P$}
\label{sec:med P}
We use Lemma~\ref{lem:ET} for the sequence of  $ R(P,Q)$ fractions 
$$
\overline{q}_p/p,  \qquad p<q,\ pq\le x, \ p \approx P, \ q \approx Q.
$$
More precisely, we choose $A =  P-1$, so that all exponential sums $S_p(a;x_p,y_p)$
with   
\begin{equation}\label{eq: yp zp}
y_p=\max\{Q,p\} \mand z_p=\min\{\xi Q,x/p\}, 
\end{equation} 
 which appear in the inequality of Lemma~\ref{lem:ET},  satisfy the condition of 
Lemma~\ref{lem:DFI-bound}.

Hence, using that $P^{1/2}+ Q^{1/2} + \min\{P,Q\} \ll Q^{1/2} + P$, we derive
\begin{equation}\label{eq: R E0 E1 E2}
 R_{\gamma}(P,Q) = \rho(PQ) R(P,Q)   +O\(E_0 +\(E_1  + E_2\) P^{o(1)} \), 
\end{equation} 
where 
$$
E_0 = R(P,Q) P^{-1}, \qquad E_1 = P^{1/2} Q, \qquad E_2 = P^{3/2}Q^{1/2}.
$$

We treat the main term in~\eqref{eq: R E0 E1 E2} via~\eqref{eq:thetaPQ}, 
getting it in the form asserted in Proposition~\ref{prop: RPQ}.

Using the trivial bound $R(P,Q)\ll PQ \cL^{-4}$,   we estimate $E_0$ as 
\begin{equation}\label{eq: Bound E0}
E_0 \ll  Q \cL^{-4}, 
\end{equation} 
which is satisfactory. We also see that for medium $P$ we have 
$$
E_1 = P^{-1/2} (PQ) \le   P^{-1/2} x \le x^{5/6} \cL^{50}
$$
and 
$$
 E_2 = P^{1/2}(PQ)^{1/2} \le P^{1/2}x^{1/2}\le x^{7/10}, 
 $$
 which are both satisfactory since
 $$
 7/10 <  5/6 < 1/2 + \gamma.
 $$
 Hence we establish Proposition~\ref{prop: RPQ} in this case.

\subsection{Large $P$}

 We now  improve the bounds of~\cite{F13} on $R_{\gamma}(P,Q)$
 for {\it large\/} $P$. We note that for $P> (2x)^{2/5}$ we have 
$$
2Q\le 2x/P \le P^{3/2}.
$$
Therefore,   in this case, both Lemma~\ref{lem:KC-bound} and Lemma~\ref{lem:Irving} apply to 
the sums $ S_p(a;y_p,z_p)$ with $y_p$ and $z_p$ as in~\eqref{eq: yp zp}. 

We use Lemma~\ref{lem:ET} for the sequence of  $R(P,Q)$ fractions 
$$
\overline{q}_p/p,  \qquad p<q,\ pq\le x, \ p \approx P, \ q \approx Q.
$$
More precisely, we choose $A =  P-1$, so that all exponential sums $S_p(a;x_p,y_p)$
with   $x_p=\max\{Q,p\}$, $y_p=\min\{\xi Q,x/p\}$, 
 which appear in inequality of Lemma~\ref{lem:ET},  satisfy the bound of 
Lemma~\ref{lem:KC-bound}. 

Hence 
\begin{equation}\label{eq: R E0 F}
 R_{\gamma}(P,Q) = \rho(PQ) R(P,Q)   +O\(E_0 + F P^{o(1)} \), 
\end{equation} 
where,  as before, $E_0 = R(P,Q) P^{-1}$ and 
$$
F= P  Q^{15/16} =  P^{1/16}  (PQ)^{15/16}  \le  P^{1/16}  x^{15/16}  . 
$$
We treat the main term in~\eqref{eq: R E0 F} and the error term $E_0$ as
in Section~\ref{sec:med P}.

Examining when $F \ll x^{1/2+\gamma}\mathcal{L}^{-6}$ we see that   we have a desired result 
provided that 
\begin{equation}\label{eq:Cond1}
P  \ll   x^{16 \gamma-7-\varepsilon}
\end{equation} 
for some fixed $\varepsilon > 0$.

Similarly, using Lemma~\ref{lem:Irving} we derive
\begin{equation}\label{eq: G1 G2 G3}
 R_{\gamma}(P,Q) = \rho(PQ) R(P,Q)   +O\(E_0+ \(G_1    +G_2   + G_3\) P^{o(1)} \), 
\end{equation} 
where,  as in~\eqref{eq: R E0 E1 E2}, we have $E_0 = R(P,Q) P^{-1}$ and 
$$
G_1  = Q^{5/8}P^{5/4}, \qquad  G_2  = Q^{9/10}P, \qquad  G_3 = Q^{13/18}P^{7/6} .
$$
Again, we deal  with the main term in~\eqref{eq: R E0 F} and the error term $E_0$ as
in Sections~\ref{sec:med P}.
 Next, writing
    \begin{align*}
& G_1=(PQ)^{5/8}P^{5/8} \ll x^{5/8}P^{5/8}, \\
& G_2=(PQ)^{9/10}P^{1/10}\ll x^{9/10}P^{1/10}, \\
& G_3 = Q^{13/18}P^{7/6}=(PQ)^{13/18}P^{4/9} \ll x^{13/18}P^{4/9},
    \end{align*}
and recalling~\eqref{eq: Bound E0}, 
we see from~\eqref{eq: G1 G2 G3} that if   $P$ satisfies
\begin{equation}\label{eq:Cond2}
 P   \ll    \min\left\{x^{(8 \gamma-1)/5 },x^{10\gamma-4},x^{(9 \gamma-2)/4}\right\}
 x^{ -\varepsilon}
 \end{equation}
 for some fixed $\varepsilon>0$ then we obtain the desired error term in Proposition~\ref{prop: RPQ}.

\subsection{Concluding the proof}
To obtain the desired result we have to make sure that any $P$ in the range  $x^{2/5} <P \le x^{1/2}$
satisfies at least one of the bounds~\eqref{eq:Cond1} or~\eqref{eq:Cond2}.  
We now consider the monotonically increasing function
$$
H(\gamma) = \max\left\{(20\gamma-6)/9,  \min\left\{(8 \gamma-1)/5, 10\gamma-4,(9 \gamma-2)/4\right\}\right\}
$$
and conclude that Proposition~\ref{prop: RPQ} holds for 
$ \gamma_0 +\eta\le \gamma \le 1/2 - \eta$, where $\gamma_0$ is the root 
of the equation 
$$
H(\gamma_0) = 1/2.
$$
One now easily verifies that $\gamma_0 = 9/20$, which concludes the proof of Proposition~\ref{prop: RPQ}
and thus also of Theorem~\ref{thm:H-gamma}.


\begin{thebibliography}{9}


\bibitem{AAHL}  A. Al-Kateeb, M. Ambrosino, H.  Hong and E. Lee,  
\textit{Maximum gap in cyclotomic polynomials}, 
 J. Number Theory, \textbf{229} (2021), 1--15, \url{https://doi.org/10.1016/j.jnt.2021.04.013}.


\bibitem{B12} B. Bzdega, \textit{Sparse binary cyclotomic polynomials}, J. Number Theory, \textbf{132} (2012), 410--413, \url{https://doi.org/10.1016/j.jnt.2011.08.002}.

\bibitem{CCLMS} O.-M. Camburu, E.-A. Ciolan, F. Luca, P. Moree and I. E. Shparlinski, \textit{Cyclotomic coefficients: gaps and
jumps}, J. Number Theory,  \textbf{163} (2016) 211--237, 
 \url{https://doi.org/10.1016/j.jnt.2015.11.020}.
 
\bibitem{Carl} L. Carlitz, \textit{The number of terms in the cyclotomic polynomial $F_{pq}(x)$}, Amer. Math. Monthly,  \textbf{73} (1966), 979--981,  \url{https://doi.org/10.2307/2314500}. 



\bibitem{DrTi} M.  Drmota and R. Tichy,
\textit{Sequences, discrepancies and applications},
Springer-Verlag, Berlin, 1997.

\bibitem {DFI} 
W. Duke, J. Friedlander  and H. Iwaniec,
\textit{Bilinear forms with Kloosterman fractions}, 
Invent. Math.,  \textbf{128} (1997) 23--43,   \url{https://doi.org/10.1007/s002220050135}. 

\bibitem{F13}  {\'E}. Fouvry, \textit{On binary cyclotomic polynomials}, Algebra Number Theory, \textbf{7} (2013), 1207--1223,   \url{https://dx.doi.org/10.2140/ant.2013.7.1207}.

\bibitem{FKM} {\'E}. Fouvry,  E. Kowalski and
  P.  Michel, `Algebraic trace functions over the primes',
\textit{Duke Math. J.} \textbf{163} (2014),  1683--1736
  \url{https://dx.doi.org/10.1215/00127094-2690587}. 

\bibitem{FS} {\'E}. Fouvry and I. E. Shparlinski, \textit{On a ternary quadratic form over primes}, Acta Arith., \textbf{150} (2011), 285--314,  \url{https://doi.org/10.4064/AA150-3-5}.


\bibitem{GalMor} Y. Gallot   and P~Moree, 
\textit{Ternary cyclotomic polynomials having a large coefficient}, 
J. Reine Angew. Math.,  \textbf{632} (2009), 105--125,   \url{https://doi.org/10.1515/crelle.2009.052}.

\bibitem{H-PM} A. Herrera-Poyatos and P. Moree, 
\textit{Coefficients and higher order derivatives of cyclotomic polynomials: Old and new}, 
Expo. Math., \textbf{40} (2022), 469--494,  \url{https://doi.org/10.1016/j.exmath.2019.07.003}. 

\bibitem{Ir} A. J. Irving, \textit{Average bounds for Kloosterman sums over primes}, Funct. Approx. Comment. Math., \textbf{51} (2014), 221--235,,  \url{https://doi.org/10.7169/facm/2014.51.2.1}.

\bibitem{IwKow} H. Iwaniec and E. Kowalski,
\textit{Analytic number theory}, Amer.  Math.  Soc.,
Providence, RI, 2004. 


\bibitem{KC} M. A. Korolev and  M. E. Changa, \textit{New estimate for Kloosterman sums with primes}, Math. Notes, \textbf{108} (2020), 94--101, \url{https://doi.org/10.1134/S0001434620070081}.

\bibitem{KMSZ} A. Kosyak, P. Moree, E. Sofos and B. Zhang,  
\textit{Cyclotomic polynomials with prescribed height and prime number theory}, 
Mathematika, \textbf{67} (2021),  214--234,  \url{https://doi.org/10.1112/mtk.12069}. 

\bibitem{KuNi}
L. Kuipers and H. Niederreiter, \textit{Uniform distribution of
sequences\/}, Wiley-Intersci., New York-London-Sydney, 1974.

\bibitem{Mig} A. Migotti, \textit{Aur Theorie der Kreisteilungsgleichung}, Z. B. der Math.-Naturwiss., Classe der Kaiserlichen Akademie der Wis-
senschaften, Wien,  \textbf{87} (1883), 7--14.


\bibitem{San} C. Sanna, 
\textit{A survey on coefficients of cyclotomic polynomials}, 
Expo. Math., \textbf{40} (2022), 469--494,  \url{https://doi.org/10.1016/j.exmath.2022.03.002}.

\bibitem{Vin} I. M. Vinogradov, \textit{The method of trigonometrical sums in the theory of numbers}, Intersci. Publ., New York, 1954.

\end{thebibliography}
\end{document}